\documentclass{amsart}
\usepackage{xcolor}
\usepackage{hyperref}
\hypersetup{
	colorlinks=true,
	citecolor=blue,
	linkcolor=blue,
	filecolor=magenta,      
	urlcolor=cyan,
}
\usepackage{zref-clever}
\zcsetup{cap,noabbrev,nameinlink=false}

\zcsetup{
  countertype={
    cthm=theorem,
    ccor=corollary,
    que=question
  }
}

\AddToHook{env/thm/begin}{\zcsetup{countertype={equation=theorem}}}
\AddToHook{env/thm/end}{\zcsetup{countertype={equation=}}}
\AddToHook{env/prop/begin}{\zcsetup{countertype={equation=proposition}}}
\AddToHook{env/prop/end}{\zcsetup{countertype={equation=}}}
\AddToHook{env/lem/begin}{\zcsetup{countertype={equation=lemma}}}
\AddToHook{env/lem/end}{\zcsetup{countertype={equation=}}}
\AddToHook{env/cor/begin}{\zcsetup{countertype={equation=corollary}}}
\AddToHook{env/cor/end}{\zcsetup{countertype={equation=}}}
\AddToHook{env/dfn/begin}{\zcsetup{countertype={equation=definition}}}
\AddToHook{env/dfn/end}{\zcsetup{countertype={equation=}}}
\AddToHook{env/rem/begin}{\zcsetup{countertype={equation=remark}}}
\AddToHook{env/rem/end}{\zcsetup{countertype={equation=}}}
\AddToHook{env/exa/begin}{\zcsetup{countertype={equation=example}}}
\AddToHook{env/exa/end}{\zcsetup{countertype={equation=}}}
\AddToHook{env/conj/begin}{\zcsetup{countertype={equation=conjecture}}}
\AddToHook{env/conj/end}{\zcsetup{countertype={equation=}}}

\usepackage[all, knot]{xy}\SelectTips{eu}{}

\usepackage{amsmath,amscd}
\usepackage{amssymb}
\usepackage{mathtools}
\usepackage{stmaryrd}
\usepackage{url}
\usepackage{tikz-cd}
\usepackage{enumitem}

\usepackage{setspace}
\usepackage{microtype}

\usepackage{times}

\usepackage{tikz}

\newtheorem{thm}[equation]{Theorem}
\newtheorem*{thm*}{Theorem}
\newtheorem*{cor*}{Corollary}
\newtheorem*{dfn*}{Definition}

\newtheorem{cthm}{Theorem}
\newtheorem{ccor}[cthm]{Corollary}

\newtheorem{cor}[equation]{Corollary}
\newtheorem{prop}[equation]{Proposition}
\newtheorem{lem}[equation]{Lemma}

\theoremstyle{definition}

\newtheorem{rem}[equation]{Remark}

\newtheorem{exa}[equation]{Example}

\newcommand{\inj}{\hookrightarrow}

\newcommand{\opn}{\operatorname}
\newcommand{\cat}[1]{\operatorname{\mathsf{#1}}}

\newcommand{\mfrak}[1]{\mathfrak{#1}}

\newcommand{\mrm}[1]{\mathrm{#1}}
\newcommand{\mbb}[1]{\mathbb{#1}}

\newcommand{\K}{\mbb{K} \hspace{0.05em}}

\newcommand{\m}{\mfrak{m}}
\newcommand{\n}{\mfrak{n}}
\newcommand{\p}{\mfrak{p}}
\newcommand{\q}{\mfrak{q}}
\newcommand{\injdim}{\operatorname{inj\,dim}}
\newcommand{\projdim}{\operatorname{proj\,dim}}

\title{K-injective complexes in commutative algebra}

\author{Liran Shaul}

\address{Department of Mathematics, University of Haifa, Haifa 3498838, Israel}

\email{shaul@math.haifa.ac.il}


\numberwithin{equation}{section}

\begin{document}

\begin{abstract}
It is an unfortunate fact that even over commutative noetherian rings,
localization of K-injective complexes need not be K-injective.
However, not all is lost.
Given a complex of injective modules $J$ over a commutative noetherian ring $A$,
we show that K-injectivity of $J$ can be detected locally:
if $(A,\m)$ is a noetherian local ring, then $J$ is K-injective over $A$
if and only if $\opn{Hom}_A(\widehat{A},J)$ is K-injective over $\widehat{A}$,
the $\m$-adic completion of $A$; and for an arbitrary commutative noetherian ring $A$,
$J$ is K-injective over $A$ if and only if $\opn{Hom}_A(A_{\m},J)$
is K-injective over $A_{\m}$ for every maximal ideal $\m$ in the singular locus of $A$.
The local result is deduced from a more general theorem on faithfully flat descent
of K-injectivity along maps to (possibly noncommutative, possibly non-noetherian) rings.
The global result is established via the local-to-global principle in tensor-triangulated geometry.
As an application, we generalize a theorem of Grothendieck on faithfully flat descent of regularity,
extending it to faithfully flat extensions by noncommutative left coherent left regular rings.
\end{abstract}

\maketitle

\setcounter{section}{-1}
\section{Introduction}

The aim of this paper is to study the behavior of K-injective complexes under standard operations in commutative algebra: localization, adic completion, and faithfully flat extensions. We start our discussion with localization.

If $A$ is an arbitrary commutative ring,
$\p \in \opn{Spec}(A)$
and $P$ is a projective (respectively flat) $A$-module,
then it is always the case that $P_{\p}$ is a projective (resp.\ flat) $A_{\p}$-module.
For flatness, the converse holds over any commutative ring. 
However, for projectivity, even for rings as nice as $\mathbb{Z}$,
the fact that $P_{\p}$ is projective over $A_{\p}$ for all $\p \in \opn{Spec}(A)$ does not imply that $P$ is projective over $A$
(see \cite[Exercise 19.12]{Eisenbud} for a concrete example over $\mathbb{Z}$).

For injective modules, in general, the situation is even worse.
The fact that $I$ is an injective $A$-module does not imply that $I_{\p}$ is injective over $A_{\p}$. However, over commutative noetherian rings,
the situation is actually even better than the projective case, and it does hold that an $A$-module $I$ is injective over $A$ if and only if $I_{\p}$ is injective over $A_{\p}$ for all $\p \in \opn{Spec}(A)$.

As all of this is well understood, a natural next step is to study the situation for complexes of $A$-modules. In that case, as is well known since the work of Spaltenstein \cite{Spal}, 
the correct generalization of these classes of modules to unbounded complexes is via the notions of K-projective, K-flat and K-injective complexes.
This paper deals with locality questions for these complexes, focusing on the K-injective case, as the K-projective and K-flat analogues turn out to be straightforward (see \zcref{section:Proj}).

Let $A$ be a commutative ring, and let $J$ be a complex of injective $A$-modules.
Recall that $J$ is called K-injective if for any acyclic complex $X$ of $A$-modules,
it holds that $\opn{Hom}_A(X,J)$ is also acyclic. If $J$ is bounded below,
this is of course automatic, but in the absence of such a boundedness condition,
a complex of injectives can fail to be K-injective, as long as the ring is not regular \cite{CET,II}. K-injective complexes are needed to define right derived functors on the entire unbounded derived category.

Unfortunately, as was observed by Neeman \cite[Example 6.5]{Neeman96} (explained in detail in \cite{Belmans} and \cite[Example 2.9]{Pos}),
even over rings as simple as $A = \mathbb{Z}[x]/(x^2)$,
it could happen that $J$ is K-injective over $A$,
but $J[s^{-1}]$ is not K-injective over $A[s^{-1}]$ for some $s \in A$.
Moreover, \cite[Proposition 2.7]{CI} constructs over $A = \mathbb{Z}[x]/(x^2)$,
a K-injective complex $J$ and a prime ideal $\p$ such that $J_{\p}$ is not K-injective.

As Positselski explained in his recent paper \cite{Pos},
the correct way to overcome this difficulty is to use colocalization instead of localization. That is, instead of passing from $J$ to $J_{\p} = J\otimes_A A_{\p}$,
one considers the colocalization $\opn{Hom}_A(A_{\p},J)$.
If $J$ is K-injective over $A$, adjunction implies it is always the case,
over any commutative ring $A$, that $\opn{Hom}_A(A_{\p},J)$ is K-injective over $A_{\p}$ for all $\p \in \opn{Spec}(A)$. The natural next question is: does the converse hold?
As a first step toward answering this question, Positselski showed \cite[Example 3.9]{Pos} that over any commutative ring $A$, if $J$ is a complex of injective $A$-modules,
and $a_1,\dots,a_n \in A$ is a finite sequence of elements that generate the unit ideal,
then $J$ is K-injective over $A$ if and only if $\opn{Hom}_A(A[a_i^{-1}],J)$ is K-injective over $A[a_i^{-1}]$ for all $1\le i \le n$.

Positselski's result deals with a finite open cover. In this paper, 
specializing to commutative noetherian rings, we show that a significantly stronger property holds.
The first main result of this paper (see \zcref{thm:main-localize}) states:
\begin{cthm}\label{cthm:loc}
Let $A$ be a commutative noetherian ring,
and let $J$ be a complex of injective $A$-modules.
Then the following are equivalent:
\begin{enumerate}
\item $J$ is K-injective over $A$.
\item For all $\p \in \opn{Spec}(A)$,
the complex $\opn{Hom}_A(A_{\p},J)$ is K-injective over $A_{\p}$.
\item For any maximal ideal $\m$ in the singular locus of $A$,
it holds that $\opn{Hom}_A(A_{\m},J)$ is K-injective over $A_{\m}$.
\end{enumerate}
\end{cthm}

Thus, not only can one check K-injectivity at prime ideals, but it is sufficient to test it after passing to maximal ideals where the ring is singular.
The proof of this result builds on Stevenson's work \cite{St1,St2} on tensor-triangulated actions of the derived category $\cat{D}(A)$ on the big singularity category $\cat{K}_{\opn{ac}}(\opn{Inj} A)$.

Next, we discuss faithfully flat extensions, and adic completion.
In the case of modules over a ring, it is straightforward that if $A \to B$ is a faithfully flat map between arbitrary commutative rings, and $F$ is an $A$-module, 
then $F$ is flat over $A$ if and only if $F\otimes_A B$ is flat over $B$.
Much less obvious, but also true, is that an $A$-module $P$ is projective over $A$ if and only if $P\otimes_A B$ is projective over $B$. This was first proved in \cite{RG},
an error in the proof was corrected in \cite{perry}, and the result was recently  formalized in Lean in \cite{shFormal}.
For injective modules, one of course cannot expect that the base change of an injective module will be injective. Indeed, taking $A= \K$ to be a field, and $B = \K[x,y]/(x^2,xy,y^2)$, we obviously have that $A \to B$ is faithfully flat, and $I = \K$ is injective over $A$, but $I\otimes_A B = B$ is not injective over $B$.
Instead, again, it is more reasonable to consider $\opn{Hom}_A(B,I)$, which trivially will be injective over $B$ as soon as $I$ is injective over $A$. 
However, the converse is false even if $(A,\m)$ is a noetherian local ring,
and $B$ is the $\m$-adic completion of $A$. See \zcref{ex:notforInj} for an example.

Passing back to unbounded complexes, it is again relatively easy to show (see \zcref{section:Proj}), that if $F$ (respectively $P$) is a complex of flat (resp.\ projective) $A$-modules, then $F$ is K-flat (resp. $P$ is K-projective) over $A$ if and only if $F\otimes_A B$ is K-flat (resp. $P\otimes_A B$ is K-projective) over $B$.

In view of the above failure for injective modules, we find it rather surprising that for complexes of injective modules, we have the following result, our second main result (see \zcref{thm:faithflat}). In it, the ring $B$ does not even have to be commutative:

\begin{cthm}\label{cthm:ff}
Let $A$ be a commutative noetherian ring.
Let $B$ be a (not necessarily commutative) central $A$-algebra which is faithfully flat over $A$.
Let $J$ be a complex of injective $A$-modules.
Then $J$ is K-injective over $A$ if and only if
$\opn{Hom}_A(B,J)$ is K-injective over $B$ as a complex of left $B$-modules.    
\end{cthm}

As an immediate corollary, we obtain:
\begin{ccor}
Let $(A,\m)$ be a noetherian local ring,
and denote by $\widehat{A}$ its $\m$-adic completion.
Let $J$ be a complex of injective $A$-modules.
Then $J$ is K-injective over $A$ if and only if $\opn{Hom}_A(\widehat{A},J)$ is K-injective over $\widehat{A}$.
\end{ccor}

As an application of the above result, we obtain the following generalization of a result of Grothendieck:
\begin{ccor}
Let $A$ be a commutative noetherian ring.
Let $B$ be a (not necessarily commutative) central $A$-algebra which is faithfully flat over $A$.
Assume that $B$ is left coherent and left regular.
Then $A$ is regular.
\end{ccor}
See \zcref{cor:regular} and the discussion preceding it for more details.

Let us briefly describe the contents of the paper.
In \zcref{sec:prel} we discuss various preliminaries about complexes of injective modules, as well as acyclic complexes,
and explain the relation between K-injectivity and the singularity category.
\zcref{sec:FF} is dedicated to faithfully flat extensions, 
and to a proof of \zcref{cthm:ff} and its corollaries.
Afterwards, in \zcref{sec:loc}, we discuss the local-to-global problem,
gather various preliminaries from the theory of tensor-triangulated actions, and use them to prove \zcref{cthm:loc}.
The final short \zcref{section:Proj} is dedicated to the analogous problems for K-projective and K-flat complexes, where the situation is much simpler.

\textbf{Acknowledgments.}

The author thanks Jindřich Novák and Amnon Yekutieli for some helpful discussions.
Parts of this work were carried out while the author worked at the department of Algebra at Charles University in Prague. 
This work has been partially supported by the grant GA~\v{C}R 26-22734S from the Czech Science Foundation.

\textbf{Statement on the use of AI.} Claude Opus 5.5 was used for light language editing.

\section{Preliminaries}\label{sec:prel}

In this section, we gather some basic results about K-injective complexes.
They are not really new, and probably known in some form, 
but as it is hard to locate exact references for them, 
we include full proofs.

We learned the next result from Jindřich Novák.
We thank him for allowing us to include it here.

\begin{prop}\label{prop:homAcyc}
Let $A$ be a ring, and let $I$ be an acyclic complex of injective left $A$-modules.
Then for any left $A$-module $M$,
any $n\in \mathbb{Z}$, and any $i\ge 1$,
there is a natural isomorphism
\[
\mrm{H}^{n+i}\left(\opn{Hom}_A(M,I)\right) = \opn{Ext}^i_A(M,\ker(d^n_I)).
\]
\end{prop}
\begin{proof}
For any $n \in \mathbb{Z}$,
truncating $I$ give rise to the complexes
\[
I' = \dots \to I^{n-2} \xrightarrow{d^{n-2}_I} I^{n-1}  \to 0 \to \dots
\]
and
\[
I'' = \dots \to 0 \to I^n \xrightarrow{d^n_I} I^{n+1} \to \dots
\]
which fit into a distinguished triangle
\[
I'' \to I \to I' \to I''[1]
\]
in $\cat{K}(A)$.
The fact that $I$ is acyclic implies that for any $n \in \mathbb{Z}$,
the sequence
\[
0 \to \ker(d^n_I) \inj I^n \xrightarrow{d^n_I} I^{n+1} \xrightarrow{d^{n+1}_I} I^{n+2} \to \dots
\]
is exact, so that $I''[n]$ is an injective resolution of $\ker(d^n_I)$.
Applying the functor $\opn{Hom}_A(M,-)$ to the above triangle yields the triangle
\[
\opn{Hom}_A(M,I'') \to \opn{Hom}_A(M,I) \to \opn{Hom}_A(M,I') \to \opn{Hom}_A(M,I'')[1]
\]
This implies that for any $i$, there is an exact sequence
\begin{equation}\label{eqn:long-ex}
\mrm{H}^{i-1}\left(\opn{Hom}_A(M,I')\right) \to \mrm{H}^{i}\left(\opn{Hom}_A(M,I'')\right)
\to \mrm{H}^{i}\left(\opn{Hom}_A(M,I)\right) \to \mrm{H}^{i}\left(\opn{Hom}_A(M,I')\right)
\end{equation}
Note that $\opn{Hom}_A(M,I')$ is concentrated in degrees $\le n-1$
so $\mrm{H}^i\left(\opn{Hom}_A(M,I')\right) = 0$ for all $i \ge n$.
Hence for $i>n$, both $\mrm{H}^i\left(\opn{Hom}_A(M,I')\right)$ and
$\mrm{H}^{i-1}\left(\opn{Hom}_A(M,I')\right)$ vanish.
The exact sequence \eqref{eqn:long-ex} then implies that
for any $i> n$ there is a natural isomorphism
\begin{align*}
\mrm{H}^i\left(\opn{Hom}_A(M,I)\right) \cong \mrm{H}^i\left(\opn{Hom}_A(M,I'')\right) \cong\\
\mrm{H}^{i-n} \left( \opn{Hom}_A(M,I''[n])\right) = 
\opn{Ext}^{i-n}_A(M,\ker(d^n_I)),
\end{align*}
as claimed.
\end{proof}

The following is a well known characterization of contractibility for acyclic complexes of injectives.

\begin{prop}\label{prop:inj-contr}
Let $A$ be a ring, 
and let $I$ be an acyclic complex of injective left $A$-modules.
Then $I$ is contractible if and only if for all $n \in \mathbb{Z}$,
the left $A$-module $\ker(d^n_I)$ is injective.
\end{prop}
\begin{proof}
Recall that by \cite[Corollary 4.3.33]{CFH},
an acyclic complex $I$ is contractible if and only if for any $n \in \mathbb{Z}$ the exact sequence
\[
0 \to \ker(d^n_I) \to I^n \to \ker(d^{n+1}_I) \to 0
\]
splits.
Since $I^n$ is an injective module,
this occurs if and only if each $\ker(d^n_I)$ is injective.
\end{proof}

We finish this section with the next proposition, which says that one can check K-injectivity of a complex of injectives on acyclic complexes of injectives.
This reduces the problem of determining whether a complex of injectives is K-injective to a test with respect to the big singularity category $\cat{K}_{\opn{ac}}(\opn{Inj} A)$.

\begin{prop}\label{prop:kinjifAcInj}
Let $A$ be a ring,
and let $I$ be a complex of injective $A$-modules.
Then $I$ is K-injective over $A$ if and only if for any acyclic complex of injective $A$-modules $J$,
it holds that $\opn{Hom}_A(J,I)$ is acyclic.
\end{prop}
\begin{proof}
Clearly, if $I$ is K-injective, this condition is satisfied.
Conversely, let $X$ be an arbitrary acyclic complex of $A$-modules.
According to \cite[Example 5]{{Kr2}} or \cite[Theorem 2.13]{Ne} (or \cite[Proposition 3.6]{Kr} if $A$ is left noetherian),
the inclusion functor $F:\cat{K}(\opn{Inj} A) \to \cat{K}(A)$ has a left adjoint
$G:\cat{K}(A) \to \cat{K}(\opn{Inj} A)$.
Since $\cat{K}(\opn{Inj} A)$ is a full subcategory of $\cat{K}(A)$,
it follows that for all $n \in \mathbb{Z}$
\begin{equation}\label{eqn:kinjhom}
\opn{Hom}_{\cat{K}(A)}(X,I[n]) = \opn{Hom}_{\cat{K}(\opn{Inj} A)}(G(X),I[n]) = 
\mrm{H}^n \left(\opn{Hom}_A(G(X),I)\right).
\end{equation}
We claim that $G(X)$ is acyclic.
To see this, let $Q:\cat{K}(A) \to \cat{D}(A)$ be the canonical functor,
let $W \in \cat{D}(A)$,
and let $W \to K$ be a K-injective resolution made of injective $A$-modules.
Then we have
\begin{align*}
\opn{Hom}_{\cat{D}(A)}(Q(G(X)),W) = 
\opn{Hom}_{\cat{K}(A)}(G(X),K) = \\
\opn{Hom}_{\cat{K}(A)}(X,K)
= \opn{Hom}_{\cat{D}(A)}(Q(X),W)
\end{align*}

Hence, by the Yoneda lemma,
we have $Q(G(X)) \cong Q(X)$ in $\cat{D}(A)$,
which shows that $G(X)$ is an acyclic complex of injectives.
Thus, by assumption,
$\mrm{H}^n \left(\opn{Hom}_A(G(X),I)\right) = 0$ for all $n$,
so \eqref{eqn:kinjhom} shows that $I$ is K-injective.
\end{proof}

\section{Faithfully flat extensions}\label{sec:FF}

In this section we study the behavior of K-injectivity with respect to faithfully flat extensions $A \to B$. An interesting feature of the results here is that we do not have to assume that the target ring $B$ is commutative.

The following result, which follows from the work of Christensen and Iyengar, will be crucial in what follows.

\begin{prop}\label{prop:CI}
Let $A$ be a commutative noetherian ring.
Let $B$ be a (not necessarily commutative) central $A$-algebra which is faithfully flat over $A$.
Given an $A$-module $J$,
it holds that $J$ is injective over $A$ if and only if
$\opn{Ext}^i_A(B,J) = 0$ for all $i\ge 1$ and $\opn{Hom}_A(B,J)$ is injective over $B$ as a left $B$-module.
\end{prop}
\begin{proof}
If $J$ is injective over $A$ the result clearly holds.
Conversely, the Ext-vanishing assumption implies that $\mrm{R}\opn{Hom}_A(B,J) \cong \opn{Hom}_A(B,J)$,
which is assumed to be an injective left $B$-module, and hence by flatness of $B$ over $A$,
it is also injective over $A$.
This means that 
\[
\injdim_A(\mrm{R}\opn{Hom}_A(B,J)) = \injdim_A(\opn{Hom}_A(B,J)) = 0.
\]
According to \cite[Corollary 3.1]{ChI},
there is an equality
\[
\injdim_A(J) = \injdim_A(\mrm{R}\opn{Hom}_A(B,J)) = 0,
\]
so that $J$ is injective over $A$.
\end{proof}

Using the above, we may now prove the main result of this section under an acyclicity assumption.
\begin{prop}\label{prop:forMain}
Let $A$ be a commutative noetherian ring.
Let $B$ be a (not necessarily commutative) central $A$-algebra which is faithfully flat over $A$.
Let $J$ be an acyclic complex of injective $A$-modules.
Then $J$ is contractible over $A$ if and only if
$\opn{Hom}_A(B,J)$ is contractible over $B$,
as a complex of left $B$-modules.
\end{prop}
\begin{proof}
As applying an additive functor to a contractible complex yields a contractible complex,
it is clear that if $J$ is contractible over $A$,
then $\opn{Hom}_A(B,J)$ is contractible over $B$.
Conversely, assume that $\opn{Hom}_A(B,J)$ is contractible over $B$.
By \zcref{prop:inj-contr},
we know that for each $n$,
the $B$-module $\ker(d^n_{\opn{Hom}_A(B,J)})$ is injective.
Since $\opn{Hom}_A(B,J)$ is assumed to be contractible over $B$,
it is also acyclic.
Applying \zcref{prop:homAcyc}, for any $i\ge 1$ and $n$ arbitrary shows that
for all $n \in \mathbb{Z}$
\begin{equation}\label{eqn:extvan}
\opn{Ext}^i_A(B,\ker(d^n_J)) = \mrm{H}^{n+i}\left(\opn{Hom}_A(B,J)\right) = 0.
\end{equation}
The fact that $J$ is an acyclic complex implies that for any $n \in \mathbb{Z}$,
there is a short exact sequence of $A$-modules
\[
0 \to \ker(d^{n-1}_J) \to J^{n-1} \to \ker(d^n_J) \to 0
\]
Applying the left exact functor $\opn{Hom}_A(B,-)$ and using \zcref{eqn:extvan} gives
the short exact sequence of $B$-modules
\[
0 \to \opn{Hom}_A(B,\ker(d^{n-1}_J)) \to \opn{Hom}_A(B,J^{n-1}) \to \opn{Hom}_A(B,\ker(d^n_J)) \to 0.
\]
Thus, the map $\opn{Hom}_A(B,J^{n-1}) \to \opn{Hom}_A(B,\ker(d^n_J))$
which is obtained by applying $\opn{Hom}_A(B,-)$ to the surjective map $J^{n-1} \to \opn{Im}(d^{n-1}_J) = \ker(d^n_J)$ remains surjective.
Now, factor the map $d^{n-1}_J:J^{n-1} \to J^n$ as
$J^{n-1} \xrightarrow{\psi} \opn{Im}(d^{n-1}_J) \xrightarrow{\phi} J^n$
where $\psi = d^{n-1}_J$ (with restricting codomain) is the surjection, and $\phi$ is the inclusion. 
We have the following chain of equalities of left $B$-modules:
\begin{align*}
\ker(d^n_{\opn{Hom}_A(B,J)}) = \opn{Im}\left(\opn{Hom}_A(B, d^{n-1}_J)\right) = \\
\opn{Im}\left(\opn{Hom}_A(B, \phi)\right) = 
\opn{Hom}_A(B,\opn{Im}(\phi)) = \opn{Hom}_A(B,\ker(d^n_J)).
\end{align*}

Here, the first equality is the fact that $\opn{Hom}_A(B,J)$ is acyclic,
the second equality follows from surjectivity of $\opn{Hom}_A(B,\psi)$,
and the third equality is because
$\phi$ is injective and $\opn{Hom}_A(B,-)$ is left exact.
Since we know that $\ker(d^n_{\opn{Hom}_A(B,J)})$ is an injective left $B$-module,
this shows that the left $B$-module $\opn{Hom}_A(B,\ker(d^n_J))$ is also injective.
Then, \eqref{eqn:extvan} and \zcref{prop:CI} imply that $\ker(d^n_J)$ is injective over $A$, which by \zcref{prop:inj-contr} implies that $J$ is contractible over $A$.
\end{proof}

In order to prove the main result of this section,
we also need to recall the following fact:
\begin{prop}\label{prop:vanish}
Let $A$ be a ring,
let $F$ be a flat $A$-module,
and let $I$ be an acyclic complex of injective $A$-modules.
Then $\opn{Hom}_A(F,I)$ is acyclic.
\end{prop}
\begin{proof}
This is a particular case of \cite[Proposition 2.8]{Verasdanis}.
See also \cite[Proposition 3.1]{Emmanouil} and \cite{Stovicek} for more general results.
\end{proof}
We remark that under the assumption that every flat module has finite projective dimension, which frequently holds in applications 
(for instance, for commutative noetherian rings of finite Krull dimension),
one may also deduce \zcref{prop:vanish} directly from \zcref{prop:homAcyc}. 

Here is the main result of this section,
which is \zcref{cthm:ff} from the introduction.

\begin{thm}\label{thm:faithflat}
Let $A$ be a commutative noetherian ring.
Let $B$ be a (not necessarily commutative) central $A$-algebra which is faithfully flat over $A$.
Let $J$ be a complex of injective $A$-modules.
Then $J$ is K-injective over $A$ if and only if
$\opn{Hom}_A(B,J)$ is K-injective over $B$ as a complex of left $B$-modules.
\end{thm}
\begin{proof}
If $J$ is K-injective over $A$,
then $\opn{Hom}_A(B,J)$ is K-injective over $B$ by \cite[Lemma 13.1.20]{YekBook}.
For the converse, suppose that $\opn{Hom}_A(B,J)$ is K-injective over $B$ as a complex of left $B$-modules.
Choose a K-injective resolution $\varphi:J \to I$,
where $I$ is a K-injective complex of injective $A$-modules.
Such a quasi-isomorphism $\varphi$ exists by \cite[Proposition 3.11]{Spal}.
Considered as a map in $\cat{K}(\opn{Inj}(A))$,
let us complete $\varphi$ to a distinguished triangle
\begin{equation}\label{eqn:distInA}
T \to J \xrightarrow{\varphi} I \to T[1]
\end{equation}
in $\cat{K}(\opn{Inj} A)$.
The fact that $\varphi$ is a quasi-isomorphism implies that the complex of injective $A$-modules $T$ is acyclic.
Applying the triangulated functor $\opn{Hom}_A(B,-):\cat{K}(\opn{Inj} A) \to \cat{K}(\opn{Inj} B)$, we obtain a distinguished triangle
\begin{equation}\label{eqn:distInB}
\opn{Hom}_A(B,T) \to \opn{Hom}_A(B,J) \xrightarrow{\opn{Hom}_A(B,\varphi)} \opn{Hom}_A(B,I) \to \opn{Hom}_A(B,T)[1]
\end{equation}
in $\cat{K}(\opn{Inj} B)$.
Since $B$ is flat over $A$, and $T$ is an acyclic complex of injective $A$-modules, by \zcref{prop:vanish}, the complex $\opn{Hom}_A(B,T)$ is also acyclic.
Hence, the triangle \eqref{eqn:distInB} implies that $\opn{Hom}_A(B,\varphi)$ is a quasi-isomorphism.
The fact that $I$ is K-injective over $A$ implies by \cite[Lemma 13.1.20]{YekBook} that $\opn{Hom}_A(B,I)$ is K-injective over $B$.
By assumption, $\opn{Hom}_A(B,J)$ is also K-injective over $B$.
As a quasi-isomorphism between K-injective complexes of injectives is a homotopy equivalence (see \cite[Corollary 5.3.24]{CFH}), 
the triangle \eqref{eqn:distInB} implies that $\opn{Hom}_A(B,T)$ is contractible over $B$.
It follows from \zcref{prop:forMain} that $T$ is contractible over $A$.
The triangle \eqref{eqn:distInA} shows that $\varphi$ is a homotopy equivalence. 
Since being K-injective is preserved by homotopy equivalences, 
we deduce that $J$ is K-injective over $A$.
\end{proof}

As the next example shows, 
the assumption in the above theorem that $J$ is a complex of injective $A$-modules is necessary. We learned it from \cite{CK}.

\begin{exa}\label{ex:notforInj}
Let $(A,\m)$ be a commutative noetherian local ring,
and assume that $\dim(A) > 0$,
that $A$ is not $\m$-adically complete,
and that $A$ is an integral domain.
Let $\widehat{A}$ be the $\m$-adic completion of $A$.
Then the completion map $A \to \widehat{A}$ is faithfully flat,
and not an isomorphism.
According to \cite[Lemma 3.3]{AEL},
it holds that $\opn{Hom}_A(\widehat{A},A) = 0$,
so that $\opn{Hom}_A(\widehat{A},A)$ is an injective $\widehat{A}$-module.
However, the assumption that $\dim(A) > 0$ implies that $A$ is not injective as an $A$-module.
\end{exa}

Applying \zcref{thm:faithflat} to the opposite ring $B^{\opn{op}}$,
we obtain:
\begin{cor}
Let $A$ be a commutative noetherian ring.
Let $B$ be a (not necessarily commutative) central $A$-algebra which is faithfully flat over $A$.
Let $J$ be a complex of injective $A$-modules.
Then $\opn{Hom}_A(B,J)$ is K-injective over $B$ as a complex of left $B$-modules
if and only if $\opn{Hom}_A(B,J)$ is K-injective over $B$ as a complex of right $B$-modules.
\end{cor}
\begin{proof}
Applying \zcref{thm:faithflat} several times,
we know that $\opn{Hom}_A(B,J)$ is K-injective over $B$ as a complex of left $B$-modules if and only if $J$ is K-injective over $A$
if and only if $\opn{Hom}_A(B^{\opn{op}},J)$ is K-injective over $B^{\opn{op}}$ as a complex of left $B^{\opn{op}}$-modules
if and only if 
$\opn{Hom}_A(B,J)$ is K-injective over $B$ as a complex of right $B$-modules.
\end{proof}

We shall now apply \zcref{thm:faithflat} in commutative algebra.

\begin{cor}\label{cor:radComp}
Let $A$ be a commutative noetherian ring,
let $I \subseteq A$ be an ideal, and suppose that $I \subseteq \opn{rad}(A)$. Denote by $\widehat{A}$ the $I$-adic completion of $A$.
Given a complex of injective $A$-modules $J$,
it holds that $J$ is K-injective over $A$ if and only if $\opn{Hom}_A(\widehat{A},J)$ is K-injective over $\widehat{A}$.
\end{cor}
\begin{proof}
By \cite[tag 00MC]{Stacks},
the assumption that $I \subseteq \opn{rad}(A)$ implies that $A \to \widehat{A}$ is faithfully flat. 
The result now follows from \zcref{thm:faithflat}.
\end{proof}

\begin{cor}
Let $(A,\m)$ be a commutative noetherian local ring,
and denote by $\widehat{A}$ the $\m$-adic completion of $A$.
Given a complex of injective $A$-modules $J$,
it holds that $J$ is K-injective over $A$ if and only if $\opn{Hom}_A(\widehat{A},J)$ is K-injective over $\widehat{A}$.
\end{cor}
\begin{proof}
This is a special case of \zcref{cor:radComp}.
\end{proof}

Recall that a functor $F:\cat{T} \to \cat{S}$ between two categories is called conservative if for any morphism $f$ in $\cat{T}$, it holds that $f$ is an isomorphism if and only if $F(f)$ is an isomorphism. If $\cat{T}, \cat{S}$ are triangulated categories, and $F$ is a triangulated functor, 
this is equivalent to $F(X) = 0 \iff X = 0$ for all $X \in \opn{Ob}(\cat{T})$.

\begin{cor}
Let $A$ be a commutative noetherian ring,
and 
let $B$ be a (not necessarily commutative) central $A$-algebra which is flat over $A$.
Then the functor $\opn{Hom}_A(B,-)$ restricts to a triangulated functor
\[
\opn{Hom}_A(B,-) : \cat{K}_{\opn{ac}}(\opn{Inj} A) \to \cat{K}_{\opn{ac}}(\opn{Inj} B)
\]
between the big singularity categories of $A$ and $B$.
If moreover, $B$ is faithfully flat over $A$,
then this functor is conservative.
\end{cor}
\begin{proof}
Clearly,
the triangulated functor $\opn{Hom}_A(B,-)$ sends $\cat{K}_{\opn{ac}}(\opn{Inj} A)$
to $\cat{K}(\opn{Inj} B)$.
By \zcref{prop:vanish}, if $J$ is an acyclic complex of injective $A$-modules,
the fact that $B$ is flat over $A$ implies that $\opn{Hom}_A(B,J)$ is acyclic.
Finally, assuming $B$ is faithfully flat over $A$,
since $\opn{Hom}_A(B,-)$ is a triangulated functor,
\zcref{prop:forMain} implies that it is conservative.
\end{proof}

The next corollary is a special case of a recent result of Positselski,
\cite[Example 3.9]{Pos},
which allows $A$ to be non-noetherian.
A far stronger result over noetherian rings will be proven in the next section (see \zcref{cor:cover}).
\begin{cor}\label{cor:pos}
Let $A$ be a commutative noetherian ring,
and let $(a_1,\dots,a_n)$ be a finite sequence of elements of $A$ that generate the unit ideal in $A$.
Let $J$ be a complex of injective $A$-modules.
Then $J$ is K-injective over $A$ if and only if $\opn{Hom}_A(A[a_i^{-1}],J)$ is K-injective over $A[a_i^{-1}]$ for all $1\le i \le n$.
\end{cor}
\begin{proof}
The assumption implies that the ring $B = \prod_{i=1}^n A[a_i^{-1}]$ is faithfully flat over $A$.
Moreover, we have an isomorphism of complexes of $B$-modules
\[
\opn{Hom}_A(B,J) \cong \prod_{i=1}^n \opn{Hom}_A( A[a_i^{-1}],J).
\]
The fact that $\opn{Hom}_A(A[a_i^{-1}],J)$ is K-injective over $A[a_i^{-1}]$ for all $1\le i \le n$ and the above isomorphism implies that $\opn{Hom}_A(B,J)$ is K-injective over $B$.
Hence \zcref{thm:faithflat} implies that $J$ is K-injective over $A$.
\end{proof}

\begin{rem}
The assumption that $A$ is noetherian in all of the above results followed from its necessity in \zcref{prop:CI}. 
This in turn was needed because of the use of \cite[Corollary 3.1]{ChI} which requires noetherianity of $A$. 
However, under some additional assumptions,
namely, that every flat $A$-module has finite projective dimension,
and that $B$ has projective dimension at most $1$ over $A$,
\cite[Theorem 4.3]{ChI} shows that the same conclusion holds even without assuming that $A$ is noetherian.
This allows the entire above discussion to be upgraded to this setting.
In particular, the assumption that $\projdim_A(B) \le 1$ holds for $B = \prod_{i=1}^n A[a_i^{-1}]$, 
which means that the above ideas can be used to recover Positselski's result even for a non-noetherian ring $A$, provided that every flat $A$-module has finite projective dimension (which holds, for example, if $A$ has cardinality at most $\aleph_n$ for some $n$).
\end{rem}

Grothendieck proved in \cite[Proposition 17.3.3]{EGA4} that if $(A,\m) \to (B,\n)$ is a faithfully flat map of commutative noetherian local rings 
and if $B$ is regular, then $A$ is regular. This easily generalizes to the non-local case, \cite[tag 07NG]{Stacks}, showing that if $A \to B$ is a faithfully flat map of commutative noetherian rings and $B$ is regular, then $A$ is regular.
The next corollary generalizes this result in two different directions,
allowing $B$ to be noncommutative, left coherent and left regular, rather than commutative noetherian and regular.

To state it, recall that a left coherent ring $B$ is called a left regular ring if any left finitely generated ideal has finite projective dimension.

\begin{cor}\label{cor:regular}
Let $A$ be a commutative noetherian ring.
Let $B$ be a (not necessarily commutative) central $A$-algebra which is faithfully flat over $A$.
Assume that $B$ is left coherent and left regular.
Then $A$ is regular (that is, $\opn{gl}\dim(A_{\p}) < \infty,$ for all $\p \in \opn{Spec}(A)$).
\end{cor}
\begin{proof}
Let $J$ be a complex of injective $A$-modules.
According to \cite[Theorem 2.1]{CET}, the fact that $B$ is left regular implies that any complex of injective left $B$-modules is K-injective.
In particular, $\opn{Hom}_A(B,J)$ is K-injective over $B$.
By \zcref{thm:faithflat},
this implies that $J$ is K-injective over $A$.
We obtain that any complex of injective $A$-modules is K-injective over $A$.
By \cite[Theorem 2.8]{II},
this implies that $A$ is regular.
\end{proof}

In particular, under the additional assumption that $B$ is commutative, we obtain:
\begin{cor}
Let $A \to B$ be a faithfully flat map of commutative rings.
Assume that $A$ is noetherian and that $B$ is coherent and regular. 
Then $A$ is regular.
\end{cor}

\section{The Local nature of K-injectivity}\label{sec:loc}

The aim of this section is to show that over a commutative noetherian ring,
one may check K-injectivity at localizations by maximal ideals.
As was shown in \zcref{prop:kinjifAcInj},
K-injectivity over a ring $A$ is related to the big singularity category
$\cat{K}_{\opn{ac}}(\opn{Inj} A)$.

Fix a commutative noetherian ring $A$.
To proceed,
we invoke the theory of actions of a tensor triangulated category on triangulated categories introduced by Stevenson in \cite{St1,St2}. 
In particular, by \cite[Section 3]{St2},
there is an action of $\cat{D}(A)$ on $\cat{K}_{\opn{ac}}(\opn{Inj} A)$,
in the sense of \cite[Definition 3.2]{St1}.
Given $M \in \cat{D}(A)$ and $J \in \cat{K}_{\opn{ac}}(\opn{Inj} A)$,
the action is given by
\[
M \odot J := F_M \otimes_A J,
\]
where $F_M$ is some K-flat resolution of $M$ over $A$,
with the property that $F_M \in \cat{K}(\opn{Flat} A)$.

Next, for any prime ideal $\p \in \opn{Spec}(A)$,
let us denote by $\opn{Inj}_{\p}(A)$ the collection of all injective $A$-modules which are isomorphic to a (possibly empty) direct sum of copies of $E(A/\p)$.
In other words, $J \in \opn{Inj}_{\p}(A)$ if and only if $J$ is an injective $A$-module 
such that $\opn{Ass}_A(J) \subseteq \{\p\}$.
We shall denote by $\Gamma_{\p}(-)$ the $\p$-torsion functor 
\[
\Gamma_{\p}(-) := \varinjlim_n \opn{Hom}_A(A/\p^n,-).
\]

The next fact is, of course, well known:
\begin{prop}\label{prop:gammapisp}
Let $A$ be a commutative noetherian ring,
let $J$ be an injective $A$-module,
and let $\p \in \opn{Spec}(A)$.
Then $\Gamma_{\p}(J_{\p}) \in \opn{Inj}_{\p}(A)$.
\end{prop}
\begin{proof}
Note that $\opn{Inj}_{\p}(A)$ is closed under direct sums,
and recall that the functor $\Gamma_{\p}((-)_{\p})$ commutes with arbitrary direct sums.
Since $A$ is noetherian, 
any injective $A$-module is a direct sum of indecomposable injectives,
and these are of the form $E(A/\q)$ where $\q \in \opn{Spec}(A)$.
Thus, we may assume that $J = E(A/\q)$ for some $\q \in \opn{Spec}(A)$.
According to the Matlis theory of injectives, \cite[Theorem 18.4]{Matsumura},
we have that $(E(A/\q))_{\p} = 0$,
unless $\q\subseteq \p$, in which case $(E(A/\q))_{\p} = E(A/\q)$.
Moreover, $\Gamma_{\p}(E(A/\q)) = 0$ unless $\p \subseteq \q$,
in which case $\Gamma_{\p}(E(A/\q)) = E(A/\q)$.
Hence, we obtain:
\[
\Gamma_{\p} \left( E(A/q)_{\p} \right) =  \left\{\begin{array}{lr} 
0 & \p \ne \q\\
E(A/\p) & \p = \q
\end{array}
\right.
\]
This shows that $\Gamma_{\p} \left( E(A/q)_{\p} \right) \in \opn{Inj}_{\p}(A)$,
as claimed.
\end{proof}

Next we specialize to the singularity category,
and define the $\p$-part of $\cat{K}_{\opn{ac}}(\opn{Inj} A)$ to be the full subcategory:
\[
W_{\p} = \{X \in \cat{K}_{\opn{ac}}(\opn{Inj} A) \mid \forall n\in \mathbb{Z}, X^n \in \opn{Inj}_{\p}(A)\}.
\]

\begin{lem}\label{lem:acyclic-at-wp}
Let $A$ be a commutative ring,
let $J$ be a complex of injective $A$-modules,
and let $\p \in \opn{Spec}(A)$.
If $\opn{Hom}_A(A_{\p},J)$ is K-injective over $A_{\p}$,
then for all $X \in W_{\p}$,
the complex $\opn{Hom}_A(X,J)$ is acyclic.
\end{lem}
\begin{proof}
If $X \in W_{\p}$,
note that each $X^n$,
being a direct sum of copies of $E(A/\p)$,
is an $A_{\p}$-module,
which implies that there is an isomorphism $X \cong X\otimes_{A_{\p}} A_{\p}$.
Hence,
by adjunction
\[
\opn{Hom}_A(X,J) \cong \opn{Hom}_A(X\otimes_{A_{\p}} A_{\p},J) \cong
\opn{Hom}_{A_{\p}}(X,\opn{Hom}_A(A_{\p},J)).
\]
Since $X$ is assumed to be acyclic,
and $\opn{Hom}_A(A_{\p},J)$ is assumed to be K-injective over $A_{\p}$,
we know that 
$\opn{Hom}_{A_{\p}}(X,\opn{Hom}_A(A_{\p},J))$ is acyclic,
so the above isomorphism shows that $\opn{Hom}_A(X,J)$ is acyclic.
\end{proof}

Following \cite[Section 3]{PSY} and \cite[Section 11.4]{CFH},
recall that if $A$ is a commutative ring,
$I \subseteq A$ is a finitely generated ideal,
and $(a_1,\dots,a_n)$ is a finite sequence of generators of $I$,
then there is a bounded complex of flat modules $C(A;a_1,\dots,a_n)$,
called the infinite dual Koszul complex, or the Čech complex,
associated to $(a_1,\dots,a_n)$.
Moreover, for any complex of $A$-modules $M$,
there is a natural $A$-linear map
$\eta_M:\Gamma_I(M) \to C(A;a_1,\dots,a_n) \otimes_A M$.

\begin{prop}\label{prop:gammaAcyclic}
Let $A$ be a commutative noetherian ring,
let $I \subseteq A$ be an ideal,
and let $J$ be a complex of injective $A$-modules.
Let $(a_1,\dots,a_n)$ be a finite sequence of generators of $I$.
Then the natural map 
$\eta_J:\Gamma_I(J) \to C(A;a_1,\dots,a_n) \otimes_A J$ is a homotopy equivalence.
In particular, if $J$ is acyclic,
then $\Gamma_{I}(J)$ is also acyclic.
\end{prop}
\begin{proof}
Step 1: Suppose first that $J$ is bounded below.
Then $\Gamma_I(J)$ and $C(A;a_1,\dots,a_n) \otimes_A J$ are also bounded below complexes of injective modules (both of these require $A$ to be noetherian), and hence, they are K-injective.
According to \cite[Lemma 13.3.13]{CFH} (the proof of \cite[Corollary 3.25]{PSY} can also be applied), 
the fact that $J$ is a complex of injectives implies that the natural map
\[
\Gamma_I(J) \to  C(A;a_1,\dots,a_n) \otimes_A J
\]
is a quasi-isomorphism.
As a quasi-isomorphism between K-injective complexes is a homotopy equivalence,
the result follows.
Step 2: Let $J$ be an arbitrary complex of injective $A$-modules.
Define the subcategory
\[
\mathcal{S} = \{X \in \cat{K}(\opn{Inj}(A)) \mid \eta_X \mbox{ is a homotopy equivalence}\}.
\]
Since $\eta$ is a natural transformation between coproduct preserving triangulated functors, it follows that $\mathcal{S}$ is a localizing subcategory of  $\cat{K}(\opn{Inj}(A))$. 
Moreover, we have seen that it contains all bounded below complexes of injective modules.
By the proof of \cite[Proposition 2.3]{Kr},
we deduce that $\mathcal{S} =  \cat{K}(\opn{Inj}(A))$.
Finally, assuming that $J$ is acyclic,
since $C(A;a_1,\dots,a_n)$ is a bounded complex of flat modules,
it is K-flat,
so the tensor product $C(A;a_1,\dots,a_n) \otimes_A J$ is also acyclic.
As $\Gamma_I(J)$ is homotopy equivalent to an acyclic complex, it is acyclic.
\end{proof}

For $\p \in \opn{Spec}(A)$, 
recall that Stevenson's framework \cite{St1} constructed from the action of $\cat{D}(A)$ on $\cat{K}_{\opn{ac}}(\opn{Inj} A)$ local cohomology and localization functors $\mathbf{\Gamma}_{\p}, L_{\p}:\cat{K}_{\opn{ac}}(\opn{Inj} A) \to \cat{K}_{\opn{ac}}(\opn{Inj} A)$. By \cite[Proposition 3.9]{St2},
they are given by $\mathbf{\Gamma}_{\p}(-) \cong C(A;a_1,\dots,a_n)\otimes -$
(where $(a_1,\dots,a_n)$ is a sequence of elements that generates $\p$),
and $L_{\p}(-) \cong (-)\otimes_A A_{\p}$.
\zcref{prop:gammaAcyclic} shows that for $J \in \cat{K}_{\opn{ac}}(\opn{Inj} A)$,
we see that $\mathbf{\Gamma}_{\p} L_{\p} ( J) \cong \Gamma_{\p}(J_{\p})$.
Thus, we will identify them in what follows.

\begin{prop}\label{prop:gammainWp}
Let $A$ be a commutative noetherian ring,
let $\p \in \opn{Spec}(A)$,
and let $X \in \cat{K}_{\opn{ac}}(\opn{Inj} A)$.
Then $\Gamma_{\p}(X_{\p}) \in W_{\p}$.
\end{prop}
\begin{proof}
For any $n \in \mathbb{Z}$,
the $A$-module $X^n$ is injective,
so by \zcref{prop:gammapisp},
we know that $\Gamma_{\p}((X^n)_{\p}) \in \opn{Inj}_{\p}(A)$.
Since localization is exact,
the fact that $X$ is acyclic implies that $X_{\p}$ is acyclic.
Since $A$ is noetherian, 
we deduce from \zcref{prop:gammaAcyclic} that $\Gamma_{\p}(X_{\p})$ is acyclic.
Thus, $\Gamma_{\p}(X_{\p}) \in \cat{K}_{\opn{ac}}(\opn{Inj} A)$,
and $\Gamma_{\p}((X^n)_{\p}) \in \opn{Inj}_{\p}(A)$,
which, by definition, says that $\Gamma_{\p}(X_{\p}) \in W_{\p}$.
\end{proof}

Recall that the singular locus of $A$ is defined to be
\[
\opn{Sing}(A) = \{\p \in \opn{Spec}(A) \mid \opn{gl}\dim(A_{\p}) = +\infty\}.
\]
We denote by $\opn{Max}(A)$ the collection of all maximal ideals in $A$.

The next lemma shows that for the question of K-injectivity,
many primes come for free.

\begin{lem}\label{lem:allprimes}
Let $A$ be a commutative noetherian ring,
and let $J$ be a complex of injective $A$-modules.
Suppose that for all maximal ideals $\m \in \opn{Sing}(A) \cap \opn{Max}(A)$,
it holds that $\opn{Hom}_A(A_{\m},J)$ is K-injective over $A_{\m}$.
Then for all $\p \in \opn{Spec}(A)$,
it holds that $\opn{Hom}_A(A_{\p},J)$ is K-injective over $A_{\p}$.
\end{lem}
\begin{proof}
If $\m \subseteq A$ is a maximal ideal such that $\m \notin \opn{Sing}(A)$,
then by definition $A_{\m}$ is a regular local ring.
According to \cite[Proposition 2.1]{II},
over a regular local ring,
any complex of injective modules is K-injective.
We deduce that for all $\m \in \opn{Max}(A)$,
it holds that $\opn{Hom}_A(A_{\m},J)$ is K-injective over $A_{\m}$.
Let $\p \in \opn{Spec}(A)$ be an arbitrary prime ideal,
and let $\m \in \opn{Max}(A)$ be some maximal ideal such that $\p \subseteq \m$.
Then the localization map $A \to A_{\p}$ factors as $A \to A_{\m} \to A_{\p}$.
Hence, by adjunction,
there is an isomorphism of complexes of $A_{\p}$-modules:
\[
\opn{Hom}_A(A_{\p},J) \cong \opn{Hom}_A(A_{\p},\opn{Hom}_A(A_{\m},J)).
\]
Since $\opn{Hom}_A(A_{\m},J)$ is K-injective over $A_{\m}$,
we deduce by \cite[Lemma 13.1.20]{YekBook} that $\opn{Hom}_A(A_{\p},\opn{Hom}_A(A_{\m},J))$ is K-injective over $A_{\p}$,
as claimed.
\end{proof}

Here is the main result of this section.
\begin{thm}\label{thm:main-localize}
Let $A$ be a commutative noetherian ring,
and let $J$ be a complex of injective $A$-modules.
Then the following are equivalent:
\begin{enumerate}
\item $J$ is K-injective over $A$.
\item For all $\p \in \opn{Spec}(A)$,
the complex $\opn{Hom}_A(A_{\p},J)$ is K-injective over $A_{\p}$.
\item For any maximal ideal $\m$ in the singular locus of $A$,
it holds that $\opn{Hom}_A(A_{\m},J)$ is K-injective over $A_{\m}$.
\end{enumerate}
\end{thm}
\begin{proof}
$(1) \implies (2)$ holds by \cite[Lemma 13.1.20]{YekBook} applied to the localization map $A \to A_{\p}$.
$(2) \implies (3)$ is trivial,
while $(3) \implies (2)$ is \zcref{lem:allprimes}.
$(2) \implies (1)$:
Assume that $\opn{Hom}_A(A_{\p},J)$ is K-injective over $A_{\p}$ for all $\p \in \opn{Spec}(A)$.
According to \zcref{lem:acyclic-at-wp},
we know that for any $\p \in \opn{Spec}(A)$,
and any $X \in W_{\p}$,
the complex $\opn{Hom}_A(X,J)$ is acyclic.
Moreover, \zcref{prop:gammainWp} tells us that for any $X \in \cat{K}_{\opn{ac}}(\opn{Inj} A)$, we have that $\Gamma_{\p}(X_{\p}) \in W_{\p}$.
Combined, this shows that for any $\p \in \opn{Spec}(A)$, and any 
$X \in \cat{K}_{\opn{ac}}(\opn{Inj} A)$,
the complex $\opn{Hom}_A(\Gamma_{\p}(X_{\p}),J)$ is acyclic.
Let us define the following full subcategory of $\cat{K}_{\opn{ac}}(\opn{Inj} A)$:
\[
\mathcal{S} := \{X \in \cat{K}_{\opn{ac}}(\opn{Inj} A) \mid \opn{Hom}_A(X,J) \mbox{ is acyclic} \}.
\]
Clearly, $\mathcal{S}$ is a triangulated subcategory of $\cat{K}_{\opn{ac}}(\opn{Inj} A)$.
Moreover, if $\{X_{\alpha}\}_{\alpha \in I}$ is a collection of objects in $\mathcal{S}$,
then
\[
\opn{Hom}_A\left(\bigoplus_{\alpha \in I} X_{\alpha}, J\right) \cong \prod_{\alpha \in I} \opn{Hom}_A\left(X_{\alpha},J\right),
\]
and as the product of acyclic complexes over a ring is acyclic
(since products are exact in $\opn{Mod}(A)$), we deduce that $\bigoplus_{\alpha \in I} X_{\alpha} \in \mathcal{S}$.
Hence, $\mathcal{S}$ is a localizing subcategory of $\cat{K}_{\opn{ac}}(\opn{Inj} A)$.

Given an arbitrary acyclic complex of injectives $X \in \cat{K}_{\opn{ac}}(\opn{Inj} A)$,
we have seen above that $\Gamma_{\p}(X_{\p}) \in \mathcal{S}$,
and since $\mathcal{S}$ is localizing,
we deduce that $\mathcal{S}$ contains the localizing subcategory generated by 
the set 
\begin{equation}\label{eqn:catGam}
\{\Gamma_{\p}(X_{\p}) \mid \p \in \opn{Spec}(A)\} \subseteq \mathcal{S}.
\end{equation}
Furthermore, by \cite[Lemma 3.10]{St2},
we know that $\mathcal{S}$ is closed under the action of $\cat{D}(A)$.
According to \cite[Corollary 5.4]{Kr},
we know that $\cat{K}_{\opn{ac}}(\opn{Inj} A)$ is compactly generated.
Hence, by \cite[Theorem 6.9]{St1},
the $\cat{D}(A)$ action on $\cat{K}_{\opn{ac}}(\opn{Inj} A)$ satisfies the local-to-global principle, in the sense of \cite[Definition 6.1]{St1},
which means that $X$ belongs to the smallest localizing subcategory of $\cat{K}_{\opn{ac}}(\opn{Inj} A)$ which contains $\{\Gamma_{\p}(X_{\p}) \mid \p \in \opn{Spec}(A)\}$ and is closed by the action of $\cat{D}(A)$.
This shows that $X \in \mathcal{S}$.
We have thus shown that for all $X \in \cat{K}_{\opn{ac}}(\opn{Inj} A)$,
it holds that $\opn{Hom}_A(X,J)$ is acyclic,
so by \zcref{prop:kinjifAcInj},
we deduce that $J$ is K-injective over $A$.
\end{proof}

For a commutative ring $A$ and an element $f \in A$,
let us set $D(f) = \{\p \in \opn{Spec}(A) \mid f\notin \p\}$.
Here is an improvement of \zcref{cor:pos}:
\begin{cor}\label{cor:cover}
Let $A$ be a commutative noetherian ring,
and let $J$ be a complex of injective $A$-modules.
Let $a_1,\dots,a_n \in A$ be such that
\[
\opn{Sing}(A) \cap \opn{Max}(A) \subseteq \bigcup_{i=1}^n D(a_i).
\]
Then $J$ is K-injective over $A$ if and only if for all $1 \le i \le n$,
it holds that 
$\opn{Hom}_A(A[a_i^{-1}],J)$ is K-injective over $A[a_i^{-1}]$.
\end{cor}
\begin{proof}
Clearly, the condition holds if $J$ is K-injective over $A$.
Conversely, assume that each $\opn{Hom}_A(A[a_i^{-1}],J)$ is K-injective over $A[a_i^{-1}]$.
By assumption, for any $\m \in \opn{Sing}(A) \cap \opn{Max}(A)$,
there exists some $i$ such that $a_i \notin \m$.
Then the localization map $A \to A_{\m}$ factors as
$A \to A[a_i^{-1}] \to A_{\m}$.
Hence, adjunction gives an isomorphism:
\[
\opn{Hom}_A(A_{\m},J) \cong \opn{Hom}_{A[a_i^{-1}]}(A_{\m}, \opn{Hom}_A(A[a_i^{-1}],J)),
\]
The same argument as in the proof of \zcref{lem:allprimes} tells us that 
$\opn{Hom}_A(A_{\m},J)$ is K-injective over $A_{\m}$,
so by \zcref{thm:main-localize},
we deduce that $J$ is K-injective over $A$.
\end{proof}

As in the case of \zcref{ex:notforInj}, the assumption that $J$ is a complex of injective modules is essential. We learned the next example from \cite[Example 3.1]{Pos}:

\begin{exa}
Over the ring of integers $\mathbb{Z}$,
we have that for any $\p \in \opn{Max}(\mathbb{Z})$,
it holds that
\[
\opn{Hom}_{\mathbb{Z}}(\mathbb{Z}_{\p},\mathbb{Z}) = 0,
\]
which is injective (hence, K-injective) over $\mathbb{Z}_{\p}$.
However, $\mathbb{Z}$ is not an injective $\mathbb{Z}$-module,
so also not K-injective over $\mathbb{Z}$.
\end{exa}

\section{K-projective and K-flat complexes}\label{section:Proj}

In this final section, for the sake of completeness, we discuss the case of K-projective and K-flat complexes. This turns out to be much easier to prove, even without a noetherian assumption. We first discuss localization, and start with the K-flat case:

\begin{prop}\label{prop:flatLocal}
Let $A$ be a commutative ring,
and let $F$ be a complex of flat $A$-modules.
Then the following are equivalent:
\begin{enumerate}
\item $F$ is K-flat over $A$.
\item For all $\p \in \opn{Spec}(A)$, it holds that $F_{\p}$ is K-flat over $A_{\p}$.
\item For all maximal ideals $\m$ in $A$, it holds that $F_{\m}$ is K-flat over $A_{\m}$.
\end{enumerate}
If, moreover, $A$ is noetherian, it is further equivalent to
\begin{enumerate}
\setcounter{enumi}{3}
\item For all maximal ideals $\m$ in $\opn{Max}(A) \cap \opn{Sing}(A)$, it holds that $F_{\m}$ is K-flat over $A_{\m}$.
\end{enumerate}
\end{prop}
\begin{proof}
$(1) \implies (2)$: If $X$ is an acyclic complex of $A_{\p}$-modules,
then $X\otimes_{A_{\p}} F_{\p} \cong X\otimes_A F_{\p} \cong (X\otimes_A F)\otimes_A A_{\p}$,
and the latter is acyclic since $F$ is K-flat over $A$, and $A_{\p}$ is flat over $A$.
$(2) \implies (3)$ is trivial. $(3) \implies (1):$
Let $X$ be an acyclic complex of $A$-modules, let $n \in \mathbb{Z}$,
and let $\m \in \opn{Max}(A)$.
Then
\[
\mrm{H}^n\left(X\otimes_A F\right)_{\m} \cong \mrm{H}^n\left(X\otimes_A F_{\m}\right) \cong
\mrm{H}^n\left(X_{\m}\otimes_{A_{\m}} F_{\m}\right) = 0,
\]
which shows that $\opn{Supp}(\mrm{H}^n\left(X\otimes_A F\right)) \cap \opn{Max}(A) = \emptyset$.
Hence, $\mrm{H}^n\left(X\otimes_A F\right)$ = 0,
which implies that $F$ is K-flat over $A$.
Assume now that $A$ is noetherian. Clearly, then, $(3) \implies (4)$.
To see that $(4) \implies (3)$,
note that if $\m \in \opn{Max}(A)$ and $\m \notin \opn{Sing}(A)$,
by \cite[Proposition 3.4]{II}, it holds that $F_{\m}$ is K-flat over $A_{\m}$.
\end{proof}

The following result is the reason why the flat case covers the projective case:
\begin{prop}\label{prop:kflatkproj}
Let $A$ be a ring,
and let $P$ be a complex of projective $A$-modules.
Then $P$ is K-projective over $A$ if and only if $P$ is K-flat over $A$.
\end{prop}
\begin{proof}
This is \cite[Theorem 7.8]{CH}.
\end{proof}

\begin{cor}
Let $A$ be a commutative ring,
and let $P$ be a complex of projective $A$-modules.
Then the following are equivalent:
\begin{enumerate}
\item $P$ is K-projective over $A$.
\item For all $\p \in \opn{Spec}(A)$, it holds that $P_{\p}$ is K-projective over $A_{\p}$.
\item For all maximal ideals $\m$ in $A$, it holds that $P_{\m}$ is K-projective over $A_{\m}$.
\end{enumerate}
If, moreover, $A$ is noetherian, it is further equivalent to
\begin{enumerate}
\setcounter{enumi}{3}
\item For all maximal ideals $\m$ in $\opn{Max}(A) \cap \opn{Sing}(A)$, it holds that $P_{\m}$ is K-projective over $A_{\m}$.
\end{enumerate}
\end{cor}
\begin{proof}
This follows from \zcref{prop:flatLocal} and \zcref{prop:kflatkproj}.
\end{proof}

Next, we discuss faithfully flat extensions.
Again, first for K-flatness.

\begin{prop}\label{prop:flatFF}
Let $A$ be a commutative ring,
and let $B$ be a (not necessarily commutative) central $A$-algebra which is faithfully flat over $A$.
Let $F$ be a complex of flat $A$-modules.
Then $F$ is K-flat over $A$ if and only if $B\otimes_A F$,
as a complex of flat left $B$-modules is K-flat over $B$.
\end{prop}
\begin{proof}
If $F$ is K-flat over $A$,
and $X$ is an acyclic complex of right $B$-modules,
the isomorphism
\[
X \otimes_B (B\otimes_A F) \cong X\otimes_A F
\]
implies that $X \otimes_B (B\otimes_A F)$ is acyclic,
so that $B\otimes_A F$ is K-flat over $B$.
Conversely,
if $B\otimes_A F$ is K-flat over $B$,
and $X$ is an acyclic complex of $A$-modules,
then flatness of $A \to B$ implies that
$X\otimes_A B$ is acyclic over $B$.
Hence, K-flatness of $B\otimes_A F$ implies that
\[
(X\otimes_A B) \otimes_B (B\otimes_A F) \cong (X\otimes_A B\otimes_A F)
\]
is acyclic.
Forgetting its $B$-structure, and considering it as a complex of $A$-modules over the commutative ring $A$, 
we know that 
\[
X\otimes_A B\otimes_A F \cong (X\otimes_A F) \otimes_A B.
\]
Since $B$ is faithfully flat over $A$, 
the fact $(X\otimes_A F) \otimes_A B$ is acyclic implies that $X\otimes_A F$ is acyclic.
Hence, $F$ is K-flat over $A$.
\end{proof}

\begin{cor}
Let $A$ be a commutative ring,
and let $B$ be a (not necessarily commutative) central $A$-algebra which is faithfully flat over $A$.
Let $P$ be a complex of projective $A$-modules.
Then $P$ is K-projective over $A$ if and only if $B\otimes_A P$,
as a complex of projective left $B$-modules is K-projective over $B$.
\end{cor}
\begin{proof}
Again, this is immediate from \zcref{prop:flatFF} and \zcref{prop:kflatkproj}.
\end{proof}

\bibliographystyle{abbrv}
\bibliography{main}

@article {CI,
    AUTHOR = {Chen, Xiao-Wu and Iyengar, Srikanth B.},
     TITLE = {Support and injective resolutions of complexes over
              commutative rings},
   JOURNAL = {Homology Homotopy Appl.},
  FJOURNAL = {Homology, Homotopy and Applications},
    VOLUME = {12},
      YEAR = {2010},
    NUMBER = {1},
     PAGES = {39--44},
      ISSN = {1532-0073,1532-0081},
   MRCLASS = {13D02},
  MRNUMBER = {2594681},
MRREVIEWER = {Jaime-Luis\ Garcia-Roig},
       DOI = {10.4310/hha.2010.v12.n1.a4},
       URL = {https://doi.org/10.4310/hha.2010.v12.n1.a4},
}

@article {II,
    AUTHOR = {Iacob, Alina and Iyengar, Srikanth B.},
     TITLE = {Homological dimensions and regular rings},
   JOURNAL = {J. Algebra},
  FJOURNAL = {Journal of Algebra},
    VOLUME = {322},
      YEAR = {2009},
    NUMBER = {10},
     PAGES = {3451--3458},
      ISSN = {0021-8693,1090-266X},
   MRCLASS = {13D05 (13D09 13H05)},
  MRNUMBER = {2568347},
MRREVIEWER = {Cristodor\ Ionescu},
       DOI = {10.1016/j.jalgebra.2009.08.006},
       URL = {https://doi.org/10.1016/j.jalgebra.2009.08.006},
}

@article {CET,
    AUTHOR = {Christensen, Lars Winther and Estrada, Sergio and Thompson,
              Peder},
     TITLE = {Acyclic complexes and regular rings},
   JOURNAL = {Mediterr. J. Math.},
  FJOURNAL = {Mediterranean Journal of Mathematics},
    VOLUME = {23},
      YEAR = {2026},
    NUMBER = {2},
     PAGES = {Paper No. 57, 14},
      ISSN = {1660-5446,1660-5454},
   MRCLASS = {16E65 (16E05 16E50)},
  MRNUMBER = {5026329},
       DOI = {10.1007/s00009-026-03055-w},
       URL = {https://doi.org/10.1007/s00009-026-03055-w},
}

@article {St1,
    AUTHOR = {Stevenson, Greg},
     TITLE = {Support theory via actions of tensor triangulated categories},
   JOURNAL = {J. Reine Angew. Math.},
  FJOURNAL = {Journal f\"ur die Reine und Angewandte Mathematik. [Crelle's
              Journal]},
    VOLUME = {681},
      YEAR = {2013},
     PAGES = {219--254},
      ISSN = {0075-4102,1435-5345},
   MRCLASS = {18E30 (18D10)},
  MRNUMBER = {3181496},
MRREVIEWER = {Pawel\ Sosna},
       DOI = {10.1515/crelle-2012-0025},
       URL = {https://doi.org/10.1515/crelle-2012-0025},
}

@article {St2,
    AUTHOR = {Stevenson, Greg},
     TITLE = {Subcategories of singularity categories via tensor actions},
   JOURNAL = {Compos. Math.},
  FJOURNAL = {Compositio Mathematica},
    VOLUME = {150},
      YEAR = {2014},
    NUMBER = {2},
     PAGES = {229--272},
      ISSN = {0010-437X,1570-5846},
   MRCLASS = {14F05 (14B05 14M10 18E30)},
  MRNUMBER = {3177268},
MRREVIEWER = {Pawel\ Sosna},
       DOI = {10.1112/S0010437X1300746X},
       URL = {https://doi.org/10.1112/S0010437X1300746X},
}

@article {Pos,
    AUTHOR = {Positselski, Leonid},
     TITLE = {Local, colocal, and antilocal properties of modules and
              complexes over commutative rings},
   JOURNAL = {J. Algebra},
  FJOURNAL = {Journal of Algebra},
    VOLUME = {646},
      YEAR = {2024},
     PAGES = {100--155},
      ISSN = {0021-8693,1090-266X},
   MRCLASS = {13B30 (14A05 18F20 18G15)},
  MRNUMBER = {4708537},
MRREVIEWER = {Majid\ Eghbali},
       DOI = {10.1016/j.jalgebra.2024.02.006},
       URL = {https://doi.org/10.1016/j.jalgebra.2024.02.006},
}

@article {Kr,
    AUTHOR = {Krause, Henning},
     TITLE = {The stable derived category of a {N}oetherian scheme},
   JOURNAL = {Compos. Math.},
  FJOURNAL = {Compositio Mathematica},
    VOLUME = {141},
      YEAR = {2005},
    NUMBER = {5},
     PAGES = {1128--1162},
      ISSN = {0010-437X,1570-5846},
   MRCLASS = {18E30 (14F05 16E30 16E65 55U35)},
  MRNUMBER = {2157133},
MRREVIEWER = {J.\ P. C. Greenlees},
       DOI = {10.1112/S0010437X05001375},
       URL = {https://doi.org/10.1112/S0010437X05001375},
}

@article {Kr2,
    AUTHOR = {Krause, Henning},
     TITLE = {Approximations and adjoints in homotopy categories},
   JOURNAL = {Math. Ann.},
  FJOURNAL = {Mathematische Annalen},
    VOLUME = {353},
      YEAR = {2012},
    NUMBER = {3},
     PAGES = {765--781},
      ISSN = {0025-5831,1432-1807},
   MRCLASS = {18E05 (18A40 18E30)},
  MRNUMBER = {2923949},
MRREVIEWER = {Jue\ Le},
       DOI = {10.1007/s00208-011-0703-y},
       URL = {https://doi.org/10.1007/s00208-011-0703-y},
}

@article {Ne,
    AUTHOR = {Neeman, Amnon},
     TITLE = {The homotopy category of injectives},
   JOURNAL = {Algebra Number Theory},
  FJOURNAL = {Algebra \& Number Theory},
    VOLUME = {8},
      YEAR = {2014},
    NUMBER = {2},
     PAGES = {429--456},
      ISSN = {1937-0652,1944-7833},
   MRCLASS = {18E15},
  MRNUMBER = {3212862},
MRREVIEWER = {Marius\ Thaule},
       DOI = {10.2140/ant.2014.8.429},
       URL = {https://doi.org/10.2140/ant.2014.8.429},
}

@article {ChI,
    AUTHOR = {Christensen, Lars Winther and Iyengar, Srikanth B.},
     TITLE = {Tests for injectivity of modules over commutative rings},
   JOURNAL = {Collect. Math.},
  FJOURNAL = {Collectanea Mathematica},
    VOLUME = {68},
      YEAR = {2017},
    NUMBER = {2},
     PAGES = {243--250},
      ISSN = {0010-0757,2038-4815},
   MRCLASS = {13C11 (13D07)},
  MRNUMBER = {3633060},
MRREVIEWER = {Jason\ Gordon\ McCullough},
       DOI = {10.1007/s13348-016-0176-0},
       URL = {https://doi.org/10.1007/s13348-016-0176-0},
}

@article {Verasdanis,
    AUTHOR = {Verasdanis, Charalampos},
     TITLE = {Colocalizing subcategories of singularity categories},
   JOURNAL = {J. Algebra},
  FJOURNAL = {Journal of Algebra},
    VOLUME = {662},
      YEAR = {2025},
     PAGES = {608--624},
      ISSN = {0021-8693,1090-266X},
   MRCLASS = {18F99 (13D09 18G80)},
  MRNUMBER = {4795668},
       DOI = {10.1016/j.jalgebra.2024.08.029},
       URL = {https://doi.org/10.1016/j.jalgebra.2024.08.029},
}

@article {Emmanouil,
    AUTHOR = {Emmanouil, Ioannis},
     TITLE = {{$\bold K$}-flatness and orthogonality in homotopy categories},
   JOURNAL = {Israel J. Math.},
  FJOURNAL = {Israel Journal of Mathematics},
    VOLUME = {255},
      YEAR = {2023},
    NUMBER = {1},
     PAGES = {201--230},
      ISSN = {0021-2172,1565-8511},
   MRCLASS = {16D40 (18N55)},
  MRNUMBER = {4619533},
       DOI = {10.1007/s11856-022-2413-4},
       URL = {https://doi.org/10.1007/s11856-022-2413-4},
}

@misc{Stovicek,
      title={On purity and applications to coderived and singularity categories}, 
      author={Jan Stovicek},
      year={2014},
      eprint={1412.1615},
      archivePrefix={arXiv},
      primaryClass={math.CT},
      url={https://arxiv.org/abs/1412.1615}, 
      note={arxiv:1412.1615}
}

@article {CK,
    AUTHOR = {Christensen, Lars Winther and K\"oksal, Fatih},
     TITLE = {Injective modules under faithfully flat ring extensions},
   JOURNAL = {Proc. Amer. Math. Soc.},
  FJOURNAL = {Proceedings of the American Mathematical Society},
    VOLUME = {144},
      YEAR = {2016},
    NUMBER = {3},
     PAGES = {1015--1020},
      ISSN = {0002-9939,1088-6826},
   MRCLASS = {13C11 (13D07)},
  MRNUMBER = {3447655},
MRREVIEWER = {Jason\ Gordon\ McCullough},
       DOI = {10.1090/proc/12791},
       URL = {https://doi.org/10.1090/proc/12791},
}

@book {CFH,
    AUTHOR = {Christensen, Lars Winther and Foxby, Hans-Bj\o rn and Holm,
              Henrik},
     TITLE = {Derived category methods in commutative algebra},
    SERIES = {Springer Monographs in Mathematics},
 PUBLISHER = {Springer, Cham},
      YEAR = {[2024] \copyright 2024},
     PAGES = {xxiii+1119},
      ISBN = {978-3-031-77452-2; 978-3-031-77453-9},
   MRCLASS = {13D09},
  MRNUMBER = {4890472},
MRREVIEWER = {Hossein\ Faridian},
       DOI = {10.1007/978-3-031-77453-9},
       URL = {https://doi.org/10.1007/978-3-031-77453-9},
}

@article {Spal,
    AUTHOR = {Spaltenstein, N.},
     TITLE = {Resolutions of unbounded complexes},
   JOURNAL = {Compositio Math.},
  FJOURNAL = {Compositio Mathematica},
    VOLUME = {65},
      YEAR = {1988},
    NUMBER = {2},
     PAGES = {121--154},
      ISSN = {0010-437X,1570-5846},
   MRCLASS = {18E25 (32C35)},
  MRNUMBER = {932640},
MRREVIEWER = {Michael\ M.\ Kapranov},
       URL = {http://www.numdam.org/item?id=CM_1988__65_2_121_0},
}

@book {YekBook,
    AUTHOR = {Yekutieli, Amnon},
     TITLE = {Derived categories},
    SERIES = {Cambridge Studies in Advanced Mathematics},
    VOLUME = {183},
 PUBLISHER = {Cambridge University Press, Cambridge},
      YEAR = {2020},
     PAGES = {xi+607},
      ISBN = {978-1-108-41933-8},
   MRCLASS = {18-02 (13D09 16E35 18G80)},
  MRNUMBER = {3971537},
MRREVIEWER = {Peder\ Thompson},
}

@article {AEL,
    AUTHOR = {Aldrich, Stephen T. and Enochs, Edgar E. and Lopez Ramos, Juan
              A.},
     TITLE = {Derived functors of {H}om relative to flat covers},
   JOURNAL = {Math. Nachr.},
  FJOURNAL = {Mathematische Nachrichten},
    VOLUME = {242},
      YEAR = {2002},
     PAGES = {17--26},
      ISSN = {0025-584X,1522-2616},
   MRCLASS = {16E30},
  MRNUMBER = {1916846},
MRREVIEWER = {Rafail\ Alizade},
       DOI = {10.1002/1522-2616(200207)242:1<17::AID-MANA17>3.0.CO;2-F},
       URL =
              {https://doi.org/10.1002/1522-2616(200207)242:1<17::AID-MANA17>3.0.CO;2-F},
}

@misc{Stacks,
    shorthand    = {Stacks},
    author       = {The {Stacks Project Authors}},
    title        = {\textit{Stacks Project}},
    howpublished = {\url{https://stacks.math.columbia.edu}},
  }

@book{Matsumura,
  title={Commutative ring theory},
  author={Matsumura, Hideyuki},
  number={8},
  year={1989},
  publisher={Cambridge university press}
}

@article {PSY,
    AUTHOR = {Porta, Marco and Shaul, Liran and Yekutieli, Amnon},
     TITLE = {On the homology of completion and torsion},
   JOURNAL = {Algebr. Represent. Theory},
  FJOURNAL = {Algebras and Representation Theory},
    VOLUME = {17},
      YEAR = {2014},
    NUMBER = {1},
     PAGES = {31--67},
      ISSN = {1386-923X,1572-9079},
   MRCLASS = {13D07 (13B35 13C12 13D09)},
  MRNUMBER = {3160712},
MRREVIEWER = {Toma\ Albu},
       DOI = {10.1007/s10468-012-9385-8},
       URL = {https://doi.org/10.1007/s10468-012-9385-8},
}

@article {EGA4,
    AUTHOR = {Grothendieck, A.},
     TITLE = {\'El\'ements de g\'eom\'etrie alg\'ebrique. {IV}. \'Etude
              locale des sch\'emas et des morphismes de sch\'emas. {I}},
   JOURNAL = {Inst. Hautes \'Etudes Sci. Publ. Math.},
  FJOURNAL = {Institut des Hautes \'Etudes Scientifiques. Publications
              Math\'ematiques},
    NUMBER = {20},
      YEAR = {1964},
     PAGES = {259},
      ISSN = {0073-8301,1618-1913},
   MRCLASS = {14.05 (13.10)},
  MRNUMBER = {173675},
MRREVIEWER = {H.\ Hironaka},
       URL = {http://www.numdam.org/item/PMIHES_1964__20__5_0},
}

@article {CH,
    AUTHOR = {Christensen, Lars Winther and Holm, Henrik},
     TITLE = {The direct limit closure of perfect complexes},
   JOURNAL = {J. Pure Appl. Algebra},
  FJOURNAL = {Journal of Pure and Applied Algebra},
    VOLUME = {219},
      YEAR = {2015},
    NUMBER = {3},
     PAGES = {449--463},
      ISSN = {0022-4049,1873-1376},
   MRCLASS = {16E05 (16D40)},
  MRNUMBER = {3279365},
MRREVIEWER = {Jan\ Trlifaj},
       DOI = {10.1016/j.jpaa.2014.05.004},
       URL = {https://doi.org/10.1016/j.jpaa.2014.05.004},
}

@misc{Belmans,
  author       = {Belmans, Pieter},
  title        = {The restriction of a homotopy-injective complex to a {Z}ariski open subset is not necessarily homotopy-injective},
  year         = {2016},
  month        = {10},
  howpublished = {Self-published note},
  url          = {https://pbelmans.ncag.info/notes/restriction-homotopy-injective.pdf},
  note         = {Available at \url{https://pbelmans.ncag.info/notes/restriction-homotopy-injective.pdf}}
}

@article {Neeman96,
    AUTHOR = {Neeman, Amnon},
     TITLE = {The {G}rothendieck duality theorem via {B}ousfield's
              techniques and {B}rown representability},
   JOURNAL = {J. Amer. Math. Soc.},
  FJOURNAL = {Journal of the American Mathematical Society},
    VOLUME = {9},
      YEAR = {1996},
    NUMBER = {1},
     PAGES = {205--236},
      ISSN = {0894-0347,1088-6834},
   MRCLASS = {18E30 (14F05)},
  MRNUMBER = {1308405},
MRREVIEWER = {Luca\ Barbieri Viale},
       DOI = {10.1090/S0894-0347-96-00174-9},
       URL = {https://doi.org/10.1090/S0894-0347-96-00174-9},
}

@article {RG,
    AUTHOR = {Raynaud, Michel and Gruson, Laurent},
     TITLE = {Crit\`eres de platitude et de projectivit\'{e}. {T}echniques de
              ``platification'' d'un module},
   JOURNAL = {Invent. Math.},
  FJOURNAL = {Inventiones Mathematicae},
    VOLUME = {13},
      YEAR = {1971},
     PAGES = {1--89},
      ISSN = {0020-9910},
   MRCLASS = {14A15 (13C10)},
  MRNUMBER = {308104},
MRREVIEWER = {Masaki Maruyama},
       DOI = {10.1007/BF01390094},
       URL = {https://doi.org/10.1007/BF01390094},
}

@misc{perry,
      title={Faithfully flat descent for projectivity of modules}, 
      author={Alexander Perry},
      year={2010},
      eprint={1011.0038},
      archivePrefix={arXiv},
      primaryClass={math.AC},
      note={arxiv:1011.0038v1},
      URL={https://arxiv.org/abs/1011.0038}, 
}

@article{shFormal,
  title={Formalization in Lean of faithfully flat descent of projectivity},
  author={Shaul, Liran},
  journal={arxiv:2603.04376},
  year={2026}
}

@book {Eisenbud,
    AUTHOR = {Eisenbud, David},
     TITLE = {Commutative algebra},
    SERIES = {Graduate Texts in Mathematics},
    VOLUME = {150},
      NOTE = {With a view toward algebraic geometry},
 PUBLISHER = {Springer-Verlag, New York},
      YEAR = {1995},
     PAGES = {xvi+785},
      ISBN = {0-387-94268-8; 0-387-94269-6},
   MRCLASS = {13-01 (14A05)},
  MRNUMBER = {1322960},
MRREVIEWER = {Matthew\ Miller},
       DOI = {10.1007/978-1-4612-5350-1},
       URL = {https://doi.org/10.1007/978-1-4612-5350-1},
}

\end{document}